\documentclass[11pt]{article}

\usepackage{amsmath}
\usepackage{amsthm}
\usepackage{amsfonts}
\usepackage{bbm}
\usepackage{amssymb}
\usepackage{graphicx}
\usepackage{color}

\usepackage{fancybox}
\newlength{\querylen}
\newcommand{\mmp}{\mathbb{P}}

\newcommand{\od}{\overset{{\rm d}}{=}}
\newcommand{\dod}{\overset{{\rm d}}{\to}}

\newcommand{\me}{\mathbb{E}}

\newcommand{\mr}{\mathbb{R}}
\newcommand{\mn}{\mathbb{N}}

\DeclareMathOperator{\1}{\mathbbm{1}}

\newtheorem{thm}{Theorem}[section]
\newtheorem{lemma}[thm]{Lemma}

\newtheorem{cor}[thm]{Corollary}

\theoremstyle{definition}

\theoremstyle{remark}

\begin{document}
\title{The collision spectrum of $\Lambda$-coalescents}

\author{Alexander Gnedin\footnote{School of Mathematical Sciences, Queen Mary, University of London, Mile End Road, London E1 4NS, UK;\ e-mail: a.gnedin@qmul.ac.uk},\ \ Alexander Iksanov\footnote{Faculty of Computer Science and Cybernetics, Taras Shevchenko National University of Kyiv, 01601 Kyiv, Ukraine; \ e-mail: iksan@univ.kiev.ua},\ \ Alexander Marynych\footnote{Faculty of Computer Science and Cybernetics, Taras Shevchenko National University of Kyiv, 01601 Kyiv, Ukraine; \ e-mail: marynych@unicyb.kiev.ua} \ \ \text{and}\ \ Martin M\"{o}hle\footnote{Mathematical Institute, Eberhard Karls University of T\"{u}bingen, 72076 T\"{u}bingen, Germany; \ e-mail: martin.moehle@uni-tuebingen.de}}

\maketitle
\begin{abstract}
\noindent
$\Lambda$-coalescents model the evolution of a coalescing system in which any number of blocks randomly sampled from the whole may merge into a larger block. For the coalescent restricted to initially $n$ singletons we study the collision spectrum $(X_{n,k}:2\le k\le n)$, where
$X_{n,k}$ counts, throughout the history of the process, the number of collisions involving exactly $k$ blocks.
Our focus is on the large $n$ asymptotics of the joint distribution of the $X_{n,k}$'s, as well as on functional limits for the bulk of the spectrum for simple coalescents. Similarly to the previous studies of the total number of collisions, the asymptotics of the collision spectrum largely depends on the
behaviour of the measure $\Lambda$ in the vicinity of $0$. In particular, for beta$(a,b)$-coalescents different types of limit distributions occur depending on whether $0<a\leq 1$, $1<a<2$, $a=2$ or $a>2$.
\end{abstract}


\hphantom{2000 Mathematics Subject Classification: }Secondary:
60C05, 60G09

\section{Introduction}\label{Sect1}
The $\Lambda$-coalescents introduced by Pitman \cite{Pitman:1999} and Sagitov \cite{Sagitov:1999}
are partition-valued Markov processes  which evolve according to the rule: if at some time the process is restricted to a partition
with $m$ separate blocks, then  any $k$ particular blocks collide and merge in one block at rate
\begin{equation}\label{lrates}
\lambda_{m,\,k}=\int_{[0,\,1]}x^k(1-x)^{m-k}x^{-2}\Lambda({\rm d}x),
\qquad 2\le k\leq m,
\end{equation}
where $\Lambda$ is a given finite measure on $[0,1]$. For a $\Lambda$-coalescent $\Pi_n=(\Pi_n(t))_{t\ge 0}$ which starts with $n$ singleton blocks
and terminates with a single block the possible states are partitions of the set $[n]:=\{1,\dots,n\}$.
Due to consistency for various $n$, one can also view $\Pi_n$ as the restriction
to $[n]$ of an infinite $\Lambda$-coalescent $\Pi_\infty=(\Pi_\infty(t))_{t\ge 0}$
with values in the set of partitions of the infinite set $\mathbb N$.

A parametric family of $\Lambda$-coalescents are the beta-coalescents with characteristic measure
\begin{equation}\label{beta(a,b)}
\Lambda({\rm d}x)= \frac{1}{{\rm B}(a,b)}  \,x^{a-1}(1-x)^{b-1}\1_{(0,1)}(x) {\rm d}x,~~a,b>0
\end{equation}
which have collision rates expressible in terms of the beta function ${\rm B}(\cdot,\cdot)$. The instance $a=b=1$, where
$\Lambda$ is the uniform distribution on $[0,1]$, is known as the Bolthausen-Sznitman coalescent.

Among the functionals characterising the speed of coalescence, the total number of collisions $X_n$
(transitions of $\Pi_n$  until absorption)  has attracted the most attention, see
\cite{Gnedin+Iksanov+Marynych:2011, Gnedin+Iksanov+Moehle:2008,Gnedin+Yakubovich:2007,Haas+Miermont:2011,Iksanov+Marynych+Moehle:2009}
and a survey paper \cite{Gnedin+Iksanov+Marynych:2014}. Here, we are interested in more delicate  properties of the coalescent process by distinguishing mergers of various sizes, that is decomposing the total number of collisions as
$$
X_n\ =\ \sum_{k=2}^n X_{n,k},\quad 2\leq k \leq n,
$$
where $X_{n,k}$ is the number of collisions resulting in a merger of $k$ blocks.
In the genealogical representation of the path of $\Pi_n$  by a rooted tree,
$X_{n,k}$ corresponds to the number of nodes of degree $k+1$. We call the collection of counts $(X_{n,k}:2\le k\le n)$
the {\it collision spectrum} of $\Pi_n$.

In many respects the collision spectrum is similar to the collection of counts of component sizes in
decomposable combinatorial structures \cite{ABT, CSP}. By this analogy, it is of  interest to look at the large-$n$ behaviour of the  first few components of the collision spectrum, as well as to identify the counts that make a major contribution to the total number of collisions.

Other types of spectra were studied for $\Lambda$-coalescents with mutations (or freeze, see e.g. \cite{DGP}),
where the terminal state of the process is the so-called allelic partition. These are the site frequency spectrum
\cite{Birkner+Blath+Eldon:2013,Blath+Cronjaeger+Christensen+Eldon+Hammer:2016,Eldon+Birkner+Blath+Freund:2015,
Spence+Kamm+Song:2016} (for the general $\Lambda$-coalescents) and the allelic partition frequency spectrum \cite{Basdevant+Goldschmidt:2008} (for the Bolthausen-Sznitman coalescent). Compared to these, the collision spectrum is simpler and better amenable to analysis for a large class of measures $\Lambda$ by the methods developed for $X_n$.

In two cases the collision spectrum is trivial. For $\Lambda$ the unit mass at $0$,
$\Pi_\infty$ is Kingman's coalescent with only binary collisions, hence $X_n=X_{n,2}=n-1$.
For $\Lambda$ the unit mass at 1, $\Pi_n$ has a sole transition, and  $X_n=X_{n,n}=1$.
In the sequel we shall assume that $\Lambda$ has no atoms at $0$ and $1$. Still, the two extremes offer the intuition that concentration of the measure $\Lambda$ near $0$ affects the abundance of collisions involving small number of blocks.

A rough quantification of the concentration of the measure $\Lambda$ near $0$ involves the first two  moments of negative order
\begin{equation}\label{eq:m_minus}
{\tt m}_{-r} \ := \ \int_{[0,\,1]}x^{-r}\Lambda({\rm d}x), ~~r=1,2
\end{equation}
which may be finite or infinite. If ${\tt m}_{-1}<\infty$ then the infinite coalescent has a nontrivial {\it dust} component, which is the
collection of singleton blocks of $\Pi_\infty(t)$, for every $t\geq 0$. Under the stronger condition ${\tt m}_{-2}<\infty$ the coalescent is {\it simple}, in the sense that transitions of $\Pi_\infty$ occur at isolated times of a Poisson process. If ${\tt m}_{-1}=\infty$ there is a further division in coalescents $\Pi_\infty$ that terminate in finite time, and coalescents that have infinitely many blocks for every $t\geq 0$.

The rest of the paper is organised as follows.
In Section \ref{Sect2} we consider the case ${\tt m}_{-2}<\infty$ of simple coalescents,
showing  that the $X_{n,k}$'s  jointly converge without scaling as $n$ grows.
Thus for every fixed $k$ the contribution of $X_{n,k}$ to $X_n$ is asymptotically negligible,
and so is for every finite collection of the values of $k$. In Section \ref{Sect3} for simple coalescents we assess
the joint contribution to $X_n$ in the form of a functional limit theorem for the process of cumulative counts
\begin{equation}\label{cumul-count}
X_n(s):=\sum_{k=2}^{\lfloor n^s\rfloor }X_{n,k},\quad s\in[0,1],
\end{equation}
where $\lfloor \cdot \rfloor$ is the floor function. In Section \ref{Sect4} we consider the case where ${\tt m}_{-2}=\infty$ but ${\tt m}_{-1}<\infty$, showing that under an assumption of regular variation each $X_{n,k}$ is asymptotic to a constant multiple of $X_n$, while for coalescents related to the gamma-type subordinators the spectrum has a nondegenerate multivariate normal distribution. In Section \ref{Sect5} we consider beta$(a,b)$-coalescents with $a\in(0,1]$ (for which ${\tt m}_{-1}=\infty$), and show that the $X_{n,k}$'s, properly centered and normalised, jointly converge to multiples of the same stable random variable.

In application to beta$(a,b)$-coalescents, our results on limit distributions of $(X_{n,k}-a_{n,k})/b_{n,k}$
are summarised in the following table:
\begin{center}
{\sc Table 1}: Limits of $(X_{n,k}-a_{n,k})/b_{n,k}$ for  ${\rm beta}(a,b)$-coalescents\\
\begin{tabular}{cccccc}
   \hline
   $a$     & $b$   & $a_{n,k}$                                        & $b_{n,k}$                                      & Limit                     & Source \\
   \hline
   $0<a<1$ & $b>0$ & $p_{k-1}^{(a)}(1-a)n$                            & $p_{k-1}^{(a)}(1-a)n^{1/(2-a)}$                & $\mathcal{S}_{2-a}$       & Thm.~\ref{main4}(i)\\
   $a=1$   & $b>0$ & $p_{k-1}^{(1)}\frac{n\log(n\log n)}{(\log n)^2}$ & $p_{k-1}^{(1)}\frac{n}{(\log n)^2}$            & $\mathcal{S}_1$           & Thm.~\ref{main4}(ii)\\
   $1<a<2$ & $b>0$ & $0$                                              & $\frac{\Gamma(k+a-2)}{k!}n^{2-a}$              & $\mathcal{E}_{2-a}$       & Thm.~\ref{main3}\\
   $a=2$   & $b>0$ & $(k\nu_1)^{-1}\log n$ & $\sqrt{\left(\frac{\nu_2}{\nu_1^3}\frac{1}{k^2}+\frac{1}{k\nu_1}\right)\log n}$   & $\mathcal{N}$             & Thm.~\ref{gammacase}\\
   $a>2$   & $b>0$ & 0 & 1 & exists & Thm.~\ref{thm:main_thm_stat}\\
   $a=3$   & $b>0$ & 0 & 1 & Poisson$(\frac{b}{k-1})$ & \cite[Thm.~3.5]{Moehle:2017+}\\
\hline
\end{tabular}
\end{center}

\vskip0.3cm
\noindent
Here, ${\cal S}_{2-a}$ denotes a random variable with a $(2-a)$-stable distribution, ${\cal E}_{2-a}$
a random variable representable as the exponential functional of a subordinator (nondecreasing
L{\'e}vy process), and ${\cal N}$ a standard normal random variable. Explanation for the scaling/centering constants will appear in a due course.

\section{Simple coalescents: convergence of the spectrum}\label{Sect2}

Let $N_n(t)$ be the number of blocks of the partition $\Pi_n(t)$.
The counting process $N_n:=(N_n(t))_{t\geq 0}$ is Markovian, starting at
$N_n(0)=n$ and decrementing from $m$ to $m-k$ at rate $\lambda_{m,k+1}$,
$1\leq k<m$. The total collision rate on $n$ blocks is
$$
\lambda_n:=\sum_{k=2}^{n}{n\choose k} \lambda_{n,k}=\int_{[0,\,1]} x^{-2}(1-nx(1-x)^{n-1}-(1-x)^n)\Lambda({\rm d}x),\quad n\geq 2,
$$
and the first transition of $\Pi_n$ is a collision of $k+1$ blocks with probability
\begin{equation}\label{p-nk}
p_{n,k}:={n\choose k+1}  \frac{  \lambda_{n,k+1}}{\lambda_n}\,,\quad 1\leq k<n.
\end{equation}
Hence by the first collision event  $N_n$ decrements from state $n$ to $n-I_n$, where $I_n$ has distribution
${\mathbb P}\{I_n=k\}=p_{n,k}$.

For the remainder of this and in the next section we assume that ${\tt m}_{-2}<\infty$. For the simple $\Lambda$-coalescents $\lambda_n\to {\tt m}_{-2}$, which is the rate of the Poisson process of collision times of $\Pi_\infty$. Let $W$ be a random variable taking values in $(0,1)$ with distribution function
\begin{equation}\label{eq:W_distribution}
\mmp\{W\leq x\}\ =\ \frac{\int_{[1-x,\,1]} y^{-2}\Lambda({\rm d}y)}{{\tt m}_{-2}},\quad x\in[0,1].
\end{equation}
Using the weak law of large numbers for binomial random variables we infer
\begin{equation}\label{bino}
\lim_{n\to\infty}\,\sum_{k=1}^{\lfloor nx\rfloor }p_{n,k}= \frac{1}{{\tt m}_{-2}}\int_{[0,\,x]} 
\ y^{-2}\Lambda({\rm d}y)=\mmp\{1-W\leq x\},\quad x\in[0,1].
\end{equation}
Thus, the variable $W$ has the intuitive meaning of the asymptotic proportion $(n-I_n)/n$ of the blocks remaining after the first collision in $\Pi_n$.

Introduce the logarithmic moments
$$
\mu:=\me |\log W|,\quad \sigma^{2}:= {\rm Var}(|\log W|)
$$
which may be finite or infinite and note that
$$
\mu=\frac{1}{{\tt m}_{-2}}\int_{[0,\,1]}|\log(1-x)|x^{-2}\Lambda({\rm d}x).
$$

Our first result states the joint convergence in distribution of the collision spectrum. We set $X_{n,k}:=0$ for $k>n$.
\begin{thm}\label{thm:main_thm_stat}
Suppose ${\tt m}_{-2}<\infty$ and that the distribution of $|\log W|$ is non-arithmetic.
\begin{itemize}
\item[\rm (i)] If  $\mu<\infty$, then there exists a non-degenerate random vector $(X_{\infty,k})_{k\geq 2}$ such that
$$
(X_{n,k})_{k\geq 2}~\Longrightarrow~ (X_{\infty,k})_{k\geq 2},\quad n\to\infty,
$$
where  $\Longrightarrow$ denotes the weak convergence in the product space $\mr^{\infty}$.
\item[\rm(ii)] If $\mu=\infty$, then
$$
X_{n,k}\overset{\mmp}{\to} 0,\quad n\to\infty
$$
for every fixed $k\geq 2$.
\end{itemize}
\end{thm}

\noindent
Note that the limit variable $X_{\infty,k}$ appearing in this theorem is not the number of $k$-collisions for the infinite coalescent.
For simple coalescents, partition $\Pi_\infty(t)$ has infinitely many blocks and every collision takes infinitely many of them.

The proof of Theorem \ref{thm:main_thm_stat} will be based on three lemmas.
\begin{lemma}\label{lem:lem1}
Let $\xi$ be a random variable with values in $[1,\infty)$. There exists a nondecreasing function $\varphi:[1,\infty)\to[0,\infty)$ slowly varying at $\infty$ such that
$$
\lim_{x\to\infty}\varphi(x)=\infty\quad\text{and}\quad \me \varphi(\xi)<\infty.
$$
\end{lemma}
\begin{proof}
This follows from the two rather obvious facts. First, there exists a positive function $\psi$ diverging to $\infty$ such that $\me \psi(\xi)<\infty$. Second, there exists a nondecreasing slowly varying function $\varphi$ diverging to $\infty$ and satisfying $\lim_{x\to\infty}\,(\varphi(x)/\psi(x))=0$. One possible construction of such a $\varphi$ can be found in the proof of Lemma 4.4 in \cite{Fay_etal:2006}.
\end{proof}
Applying Lemma \ref{lem:lem1} to $\xi=1/(1-W)$ gives the following.
\begin{cor}\label{cor:cor1}
If ${\tt m}_{-2}<\infty$ then there exists a nondecreasing function $\varphi:[1,\infty)\to[0,\infty)$ slowly varying at $\infty$ such that
$\lim_{x\to\infty}\varphi(x)=\infty$ and
\begin{equation}\label{int-phi}
\quad \me \varphi(1/(1-W))=\int_{[0,\,1]}
y^{-2}\varphi(y^{-1})\Lambda({\rm d}y)<\infty.
\end{equation}
\end{cor}

\begin{lemma}\label{lem:lem2}
Let $(a_k)_{k\in\mn}$ be a sequence defined recursively as follows:
$$
a_1=a_2=\cdots=a_\ell=0,\quad a_n=p_{n,\ell}\,\varphi(n)+\sum_{k=1}^{n-1}p_{n,k}a_{n-k},\quad n>\ell,
$$
where $\ell\in\mn$ is fixed, and $\varphi$ is a positive function slowly varying at $\infty$ such that {\rm \,(\ref{int-phi})\,} holds.
Then the sequence $(a_n)$ is bounded.
\end{lemma}
\begin{proof}
Choose $\theta\in(0,1)$ such that $\int_{[1-\theta,\,1]}y^{-2}\Lambda({\rm d}y)>0$.
We first prove by induction that there is a constant $C_\ell>0$ such that
$$
a_n\leq C_\ell\sum_{m=\ell+1}^{n}\frac{p_{m,\ell}\varphi (m)}{m},\quad n>\ell.
$$
By adjusting $C_\ell$ if necessary it is enough to show this  for $n>n_0$, where $n_0$ is any fixed integer. We have
\begin{align*}
a_n&= p_{n,\ell}\,\varphi(n)+\sum_{k=1}^{n-1}p_{n,n-k}a_k=p_{n,\ell}\,\varphi(n)+\sum_{k=\ell+1}^{n-1}p_{n,n-k}a_k\\
&\leq p_{n,\ell}\,\varphi(n)+C_\ell\sum_{k=\ell+1}^{n-1}p_{n,n-k}\sum_{m=\ell+1}^k\frac{p_{m,\ell}\,\varphi(m)}{m}=p_{n,\ell}\,\varphi (n)+C_\ell\sum_{m=\ell+1}^{n-1}\frac{p_{m,\ell}\,\varphi(m)}{m}\sum_{k=1}^{n-m}p_{n,k}.
\end{align*}
We need to check that for some $C_\ell>0$ the inequality
$$
p_{n,\ell}\varphi(n)+C_\ell\sum_{m=\ell+1}^{n-1}\frac{p_{m,\ell}\,\varphi(m)}{m}\sum_{k=1}^{n-m}p_{n,k}\leq C_\ell \sum_{m=\ell+1}^{n}\frac{p_{m,\ell}\,\varphi(m)}{m}
$$
holds for large enough $n$. This is equivalent to
\begin{eqnarray}\label{eq:lem2_ineq1}
p_{n,\ell}\,\varphi(n)\leq C_\ell \sum_{m=\ell+1}^{n}\frac{p_{m,\ell}\,\varphi(m)}{m}\sum_{j=1}^{m-1}p_{n,n-j}
\end{eqnarray}
for large enough $n$. We have
\begin{align}\label{eq:lem2_ineq2}
\sum_{m=\ell+1}^n \frac{p_{m,\ell}\,\varphi(m)}{m}\sum_{j=1}^{m-1}p_{n,n-j}&
\geq \sum_{m=\lfloor\theta n\rfloor}^n \frac{p_{m,\ell}\varphi(m)}{m}\sum_{j=1}^{m-1}p_{n,n-j} \nonumber  \\
&\geq \inf_{\lfloor\theta n\rfloor \leq m\leq n}\left({m\choose \ell+1}\frac{\varphi(m)}{m\lambda_m}\right)\lambda_{n,\ell+1}\sum_{k=\lfloor\theta n\rfloor}^{n}\sum_{j=1}^{k-1}p_{n,n-j}.
\end{align}
Since $\lambda_n\to {\tt m}_{-2}\in(0,\infty)$, the sequence $\left({m\choose \ell+1} \frac{\varphi(m)}{m\lambda_m}\right)$ is regularly varying with index $\ell$. Therefore,
\begin{equation}\label{eq:lem2_ineq3}
\inf_{\lfloor\theta n\rfloor\leq m\leq n}\left({m\choose \ell+1} \frac{\varphi(m)}{m\lambda_m}\right)\geq C'_\ell {n\choose \ell+1}\frac{\varphi(n)}{n}
\end{equation}
for some $C'_\ell>0$ and large enough $n$, by \cite{BGT} (Theorem 1.5.3). Finally,
\begin{equation}\label{eq:lem2_ineq4}
\liminf_{n\to\infty}\frac{1}{n}\sum_{k=\lfloor\theta n\rfloor}^{n}\sum_{j=1}^{k-1}p_{n,n-j}>0
\end{equation}
is a consequence of
$$
\liminf_{n\to\infty}\frac{1}{n}\sum_{k=\lfloor\theta n\rfloor}^{n}\sum_{j=1}^{k-1}p_{n,n-j}\geq (1-\theta)\liminf_{n\to\infty}\sum_{j=1}^{\lfloor\theta n\rfloor -1}p_{n,n-j}>0,
$$
where the last inequality follows from \eqref{bino} and our choice of $\theta$.

Inequality \eqref{eq:lem2_ineq1} now follows from \eqref{eq:lem2_ineq2}, \eqref{eq:lem2_ineq3}, \eqref{eq:lem2_ineq4}, definition (\ref{p-nk}) and the convergence of $\lambda_n$.

Next, we argue that the series $\sum_{m\leq \ell+1}
\frac{p_{m,\ell}\varphi(m)}{m}$ converges. Indeed, since $\lambda_n=O(1)$ we can find some $C''_\ell>0$ such that
\begin{align*}
\sum_{m\geq \ell+1}
\frac{p_{m,\ell}\,\varphi(m)}{m}&\leq C''_\ell \sum_{m\geq \ell+1}
{m^\ell \varphi(m)\lambda_{m,\ell+1}}\\
&=C''_\ell\int_{[0,\,1]} x^{\ell-1}(1-x)^{-\ell-1}\left(\sum_{m\geq \ell+1}
m^\ell\, \varphi(m)(1-x)^m\right)\Lambda({\rm d}x).
\end{align*}
By dominated convergence the integrand is bounded in some left vicinity of $1$.
By the Abelian theorem for power series, see, e.g., \cite{Feller:1973} (Chapter XIII.5, Theorem 5),
$$
\sum_{m\geq \ell+1} m^\ell \varphi (m)(1-x)^m~\sim~ {\rm const}\cdot x^{-\ell-1}\,\varphi(1/x),\quad x\to 0+
$$
which in combination with \eqref{int-phi} shows that the integral converges in some right vicinity of $0$.
\end{proof}

Now we are in position to prove Theorem \ref{thm:main_thm_stat}.

\begin{proof}[Proof of \,{\rm Theorem \ref{thm:main_thm_stat}}]
\noindent
{\sc Case $\mu<\infty$}. According to the Cram\'{e}r-Wold device it is enough to show that
$$
Z^{(N)}_n:=\sum_{k=2}^{N}\alpha_k X_{n,k}\dod \sum_{k=2}^{N}\alpha_k X_{\infty,k},\quad n\to\infty
$$
for every choice of  $N\geq 2$ and nonnegative reals $\alpha_2,\ldots,\alpha_N$.
Observe the stochastic recurrence
$$
Z^{(N)}_1=0,\quad Z^{(N)}_n\od \sum_{k=2}^{N}\alpha_k\1_{\{I_n=k-1\}}+\widehat{Z}^{(N)}_{n-I_n},\quad n\geq 2,
$$
where $I_n$ is the size of the first decrement of $N_n$, $\widehat{Z}_m\stackrel{{\rm d}}{=}Z_m$, and $I_n$ is independent of $\widehat{Z}_m$'s.
Let $A$ be the span of $\alpha_2,\dots,\alpha_N$ with integer weights. For $z\in A$ we get
\begin{eqnarray*}
&&\hspace{-0.6cm}\mmp\{Z^{(N)}_n=z\}=\sum_{m=1}^{n-1} p_{n,n-m}\mmp\left\{Z_m^{(N)}=z-\sum_{k=2}^{N}\alpha_k\delta_{n-m,k-1}\right\}\\
&&\hspace{-0.3cm}=\sum_{m=1}^{n-1}p_{n,n-m}\,\mmp\{Z_m^{(N)}=z\}+\sum_{m=1}^{n-1}p_{n,n-m}\left(\mmp\left\{Z_m^{(N)}=z-\sum_{k=2}^{N}\alpha_k\delta_{n-m,k-1}\right\}-\mmp\{Z_m^{(N)}=z\}\right)\\
&&\hspace{-0.3cm}=\sum_{m=1}^{n-1}p_{n,n-m}   \mmp\{Z_m^{(N)}=z\}+\sum_{m=n-N+1}^{n-1}p_{n,n-m} \left(\mmp\left\{Z_m^{(N)}=z-\alpha_{n-m+1}\right\}-\mmp\{Z_m^{(N)}=z\}\right)\\
&&\hspace{-0.3cm}=\sum_{m=1}^{n-1}p_{n,n-m}\mmp\{Z_m^{(N)}=z\}+\sum_{m=2}^{N}p_{n,m-1} \left(\mmp\left\{Z_{n-m+1}^{(N)}=z-\alpha_{m}\right\}-\mmp\{Z_{n-m+1}^{(N)}=z\}\right)\\
&&\hspace{-0.3cm}=:\sum_{m=1}^{n-1}p_{n,n-m}\mmp\{Z_m^{(N)}=z\}+s^{(N)}_{n,z},
\end{eqnarray*}
where $\delta_{x,y}$ is the Kronecker delta.

For  $k\leq n$ let $h(n, k)$ be the probability that the block-counting process $N_n$
ever visits state $k$. Then
$$
\mmp\{Z^{(N)}_n=z\}=\delta_{z,0}+\sum_{k=2}^{n}h(n,k)s^{(N)}_{k,z},\quad z\in A.
$$
For $L_n$ the number of blocks involved in the last collision we have
$$
{\mmp\{L_n=k\}}=h(n,k) {p_{k,k-1}}.
$$
Thus,
$$
\mmp\{Z^{(N)}_n=z\}=\delta_{z,0}+\me f^{(N)}_z(L_n),\quad z\in A,
$$
where
$$
f^{(N)}_z(m):=\frac{s^{(N)}_{m,z}}{p_{m,m-1}},\quad m\geq 2.
$$
By \cite{Kersting+Schweinsberg+Wakolbinger:2017+} (Theorem 2)
$L_n$ converges in distribution to a random variable $L_\infty$. Hence,
$$
\lim_{n\to\infty}\mmp\{Z^{(N)}_n=z\}=\delta_{z,0}+\me f^{(N)}_z(L_{\infty}),\quad z\in A,
$$
provided that the family $(f^{(N)}_z(L_{n}))_{n\in\mn}$ is uniformly integrable. In view of
$$
|f^{(N)}_z(m)|\leq \frac{\sum_{k=1}^{N-1}p_{m,k}}{p_{m,m-1}},\quad m\geq 2,
$$
it suffices to check uniform integrability of $(g^{(k)}_1(L_{n}))_{n\in\mn}$ for every fixed $k\in\mn$, where $g_1^{(k)}(m):=p_{m,k}/p_{m,m-1}$. According to
the Vall\'{e}e--Poussin criterion in the form given in
\cite{Shiryaev:2007} (Lemma 3, p.~ 267) the latter is secured by
\begin{equation}\label{eq:ui1}
\sup_{n\in\mn}\me \left(g_1^{(k)}(L_{n}) \varphi (\log^{+} g_1^{(k)}(L_{n}))\right)<\infty
\end{equation}
with a nondecreasing slowly varying function $\varphi$ as given in Corollary \ref{cor:cor1}.
We will show even more, namely that
\begin{equation}\label{eq:ui2}
\sup_{n\in\mn}\me g_2^{(k)}(L_{n}) <\infty,
\end{equation}
where $g_2^{(k)}(m):=\frac{p_{m,k}}{p_{m,m-1}}\varphi  (|\log p_{m,m-1}|)$. The expectation under the supremum can be written as
$$
\me g_2^{(k)}(L_{n})=\sum_{m=2}^{n}h(n,m)p_{m,k} \varphi(|\log p_{m,m-1}|).
$$
We use Jensen's inequality
\begin{eqnarray*}
-\log  p_{m,m-1}&=&-\log \frac{\int_{[0,\,1]} x^{m-2}\Lambda({\rm d}x)}{\Lambda([0,1])}-\log\Lambda([0,1]) +\log \lambda_m\\&\leq&
(m-2)\frac{\int_{[0,\,1]}|\log x|\Lambda({\rm d}x)}{\Lambda([0,1])}-\log\Lambda([0,1]) +\log \lambda_m
\end{eqnarray*}
in combination with $\lim_{m\to\infty}\,\lambda_m={\tt m}_{-2}$ to conclude that
$$
-\log p_{m,m-1} \leq {\rm const}\cdot m,\quad m\geq 2
$$
and thereupon
$$
\me g_2^{(k)}(L_{n})\leq {\rm const} \sum_{m=2}^{n}h(n,m) p_{m,k}  \varphi(m)
$$
by monotonicity of $\varphi$. The last sum is uniformly bounded by Lemma \ref{lem:lem2}.

{\sc Case $\mu=\infty$.} This follows immediately from the observation
$$
\lim_{n\to\infty}h(n,m)=0,\quad m\in\mn
$$
which is a consequence of \cite{Kersting+Schweinsberg+Wakolbinger:2017+} (Theorem 3).
\end{proof}

\noindent \textbf{Remark}. Under the assumptions of Theorem \ref{thm:main_thm_stat} we also have
convergence of all joint moments: for every $N\geq 2$ and $m_2,m_3,\ldots,m_N\geq 0$,
$$
\lim_{n\to\infty}\me(X^{m_2}_{n,2}X^{m_3}_{n,3}\cdots X^{m_N}_{n,N})=\me(X^{m_2}_{\infty,2}X^{m_3}_{\infty,3}\cdots X^{m_N}_{\infty,N}).
$$
In order to see this it is enough to check that
\begin{equation}\label{eq:moments_simple_coal}
\sup_{n\geq 2}\me X^m_{n,k}<\infty
\end{equation}
for all $k\geq 2$ and all $m\in\mathbb{N}$. Uniform integrability of the family
$(X^{m_2}_{n,2}X^{m_3}_{n,3}\cdots X^{m_N}_{n,N})_{n\geq 2}$ will then follow from H\"{o}lder's inequality. While condition \eqref{eq:moments_simple_coal} for $m=1$ follows from the recursion
$$
\me X_{n,k}=p_{n,k-1}+\sum_{j=1}^{n-k}p_{n,n-j}\me X_{j,k},\quad n\geq 2,
$$
for $m\geq 2$ it is checked by induction. We omit the details.

Theorem \ref{thm:main_thm_stat} is a pure existence result. Nevertheless, for beta$(3,b)$-coalescents it is possible to describe the asymptotic joint distribution of the collision spectrum explicitly. This result is strikingly similar to the classic Poisson limit for the small-block counts of Ewens' partitions \cite{ABT}, although we do not see a direct connection.

\noindent \textbf{Example} (Theorem 3.5 in \cite{Moehle:2017+}). Suppose $\Lambda$ is a {\rm beta}$(3,b)$-distribution. Then
$$
(X_{n,k})_{k\geq 2}~\Longrightarrow~ (X_{\infty, k})_{k\geq 2},\quad n\to\infty,
$$
where $(X_{\infty, k})_{k\geq 2}$ are independent with $X_{\infty, k}\od {\rm Poisson}(b/(k-1))$.

\section{Simple coalescents: functional limits}\label{Sect3}

A consequence of Theorem \ref{thm:main_thm_stat} is that the contribution of  $X_{n,k}$ to $X_n$ for every fixed $k$ remains bounded as $n$ grows.
In this section we prove functional limit theorems for the process of cumulative counts $(X_n(s))_{s\in[0,1]}$ defined by \eqref{cumul-count}.
The cases of finite and infinite $\mu$ are treated separately (Theorems \ref{main1} and \ref{main2}).

The process $(X_n(s))_{s\in[0,1]}$ has paths belonging to the Skorohod space $D[0,1]$  of c\`{a}dl\`{a}g functions.
We endow $D[0,1]$ with either the $J_1$- or the $M_1$-topology and denote the associated weak convergence of probability measures  by $\overset{J_{1}}{\Longrightarrow}$ and $\overset{M_{1}}{\Longrightarrow}$, respectively.
\begin{thm}\label{main1}
Assume  ${\tt m}_{-2}<\infty$ and that for some $a>0$
\begin{equation}\label{eq:m_minus_2_log_finite}
\me |\log (1-W)|^{a}=    \frac{1}{{\tt m}_{-2}}\int_{[0,\,1]}
|\log x|^{a} x^{-2}\Lambda({\rm d}x)<\infty.
\end{equation}
Suppose further that $\mu<\infty$ and set
\begin{align*}
u_{n}(s)&\ :=\ \mu^{-1}\int_{(1-s)\log n}^{\log n}\mmp\{|\log(1-W)|\le y\}\ {\rm d}y,\\
v_{n}(s)&\ :=\ \mu^{-1}\int_{(1-s)\log n}^{\log n}\mmp\{|\log(1-W)|> y\}\ {\rm d}y\ =\ \mu^{-1}s\log n - u_{n}(s)
\end{align*}
for $s\in [0,1]$.
\begin{itemize}
\item[\rm (A1)] If $\sigma^2<\infty$, then as $n\to\infty$
\begin{equation*}\label{clt_finite_variance}
\left(\frac{X_{n}(s)-u_{n}(s)}{\sqrt{\mu^{-3}\sigma^2\log
n}}\right)_{s\in[0,1]}\ \overset{J_{1}}{\Longrightarrow}\
(B(s))_{s\in[0,1]},
\end{equation*}
where $(B(s))_{s\in [0,1]}$ is a standard Brownian motion.
\item[\rm (A2)] If $\sigma^2=\infty$ and
\begin{equation*}\label{normalization_sv_sec_moment}
\me (\log W)^{2}\1_{\{|\log W|\le x\}}\ \sim \ \ell(x),\quad x\to\infty,
\end{equation*}
for some function $\ell$ slowly varying at infinity, then as $n\to\infty$
\begin{equation*}\label{clt_sv_sec_moment}
\left(\frac{X_{n}(s)-u_{n}(s)}{\mu^{-3/2}c(\log n)}\right)_{s\in[0,1]}\ \overset{J_{1}}{\Longrightarrow}\ (B(s))_{s\in[0,1]},
\end{equation*}
where $c$ is a positive function satisfying $\lim_{x\to\infty}
(c(x))^{-2} x\ell(c(x))=1$.
\item[\rm (A3)] If
\begin{equation}\label{normalization_rv_sec_moment}
\mmp\{|\log W|>x\}\ \sim\ x^{-\alpha}\ell(x),\quad x\to\infty,
\end{equation}
for some $\alpha\in (1,2)$ and some $\ell$ slowly varying at
infinity, then as $n\to\infty$
\begin{equation*}\label{clt_rv_sec_moment}
\left(\frac{X_{n}(s)-u_{n}(s)}{\mu^{-(\alpha+1)/\alpha}c(\log
n)}\right)_{s\in [0,1]}\ \overset{M_{1}}{\Longrightarrow}\
(S_{\alpha}(s))_{s\in[0,1]},
\end{equation*}
where $c$ is a positive function satisfying
$\lim_{x\to\infty} (c(x))^{-\alpha} x\ell(c(x))=1$ and
$(S_{\alpha}(s))_{s\in [0,1]}$ is a spectrally negative $\alpha$-stable L\'{e}vy process such that $S_{\alpha}(1)$ has the characteristic function
\begin{equation}\label{cf_st}
u\ \mapsto\ \exp\left\{-|u|^\alpha
\Gamma(1-\alpha)\left(\cos(\pi\alpha/2)+i\sin(\pi\alpha/2)\, {\rm
sgn}(u)\right)\right\}, \ u\in\mr
\end{equation}
with $\Gamma$ being the gamma function.
\end{itemize}
\end{thm}

\vskip0.3cm
\noindent
Without moment condition \eqref{eq:m_minus_2_log_finite} the conclusions of Theorem \ref{main1}
are still valid in the weaker sense of the convergence of the finite-dimensional distributions of $(X_{n}(s))_{s\in[0,1]}$, see \cite{Alsmeyer+Iksanov+Marynych:2017} (Remark 2.3).

\vskip0.3cm

\begin{thm}\label{main2}
Assume ${\tt m}_{-2}<\infty$.
If relation \eqref{normalization_rv_sec_moment} holds with $\alpha\in (0,1)$, then as $n\to\infty$
\begin{equation}\label{clt_{i}nf_exp}
\bigg(\frac{\ell(\log n)X_{n}(s)}{(\log n)^{\alpha}}\bigg)_{s\in[0,1]}\ \overset{J_{1}}{\Longrightarrow}\ \big(W^{\leftarrow}_{\alpha}(1)-W^{\leftarrow}_{\alpha}(1-s)\big)_{s\in[0,1]},
\end{equation}
where $W^{\leftarrow}_{\alpha}(s):=\inf\{y\ge 0:W_{\alpha}(y)>s\}$ for $s\ge 0$ and $(W_{\alpha}(y))_{y\ge 0}$ is an $\alpha$-stable subordinator (nondecreasing L\'{e}vy
process) with the Laplace exponent $-\log \me (-zW_{\alpha}(1))=\Gamma(1-\alpha)z^{\alpha}$, $z\ge 0$.
\end{thm}

For simple $\Lambda$-coalescents  each partition $\Pi_\infty(t)$  has a dust component.
Therefore, a coupling between the infinite coalescent and a subordinator (which is a compound Poisson process in the case ${\tt m}_{-2}<\infty$)
\cite{Gnedin+Iksanov+Marynych:2011,Gnedin+Iksanov+Moehle:2008} can be applied to relate $N_n$
with a simpler counting process derived from the dust component. We briefly summarise the combinatorial part of this connection.

Consider an {\it extended} coalescent, which is a process on partitions of $[n]$ where
every initial singleton block $\{1\},\dots, \{n\}$ is regarded as {\it primary} and every other block as {\it secondary}.
Whenever the partition  has $m$ blocks (some of which being primary and some secondary), every $k$-tuple of the blocks is merging
in one block at rate $\lambda_{m,k}$ (for $2\leq k\leq m$), and every primary block transforms  into secondary at rate $\lambda_{m,1}$.
The rate $\lambda_{m,1}$ is defined by  formula \eqref{lrates} with $k=1$, and we have $0<\lambda_{m,1}<\infty$
because ${\tt m}_{-1}<\infty$. Let  $N^\ast_n(t)$ be the number of primary blocks at time $t$, and
let $K_{n,k}$ be the number of decrements of size $k\in [n]$ of the process $N^\ast_n:=(N^{\ast}_n(t))_{t\ge 0}$.  With the extended coalescent we may associate a {\it partition of $[n]$ by the first event}, which has two integers $i$ and $j$ in the same block  if
the first event involving $\{i\}$ (collision or transformation into secondary block) is the same as for $\{j\}$; the number of $k$-blocks of this partition is then $K_{n,k}$. With the natural time-ordering of the blocks, the partition by the first event is a regenerative composition as introduced in
\cite{RCS}. Figure \ref{coal1} represents a realisation of an extended coalescent with seven initial blocks.

Let $Z_n$ be the number of secondary blocks emerging throughout the history of the extended coalescent.
Clearly, $Z_n\leq K_{n,1}+X_n$, since a secondary block results from either a collision or a transformation of a primary block,
while the number of transformations into a secondary block does not exceed $K_{n,1}$. Choosing $c_0>\mu^{-1}$ we have
\begin{equation}\label{eq:c_n_estimate}
\mmp\{Z_n>c_0\log n\}\to 0,\quad n\to\infty,
\end{equation}
as follows from \cite{Gnedin+Iksanov+Marynych:2011} (Proposition 5.1 and Theorem 5.1) in conjunction
with \cite{Gnedin+Iksanov+Marynych:2010} (Corollary 1.1).

\begin{figure}
\begin{center}
\includegraphics[scale=0.5]{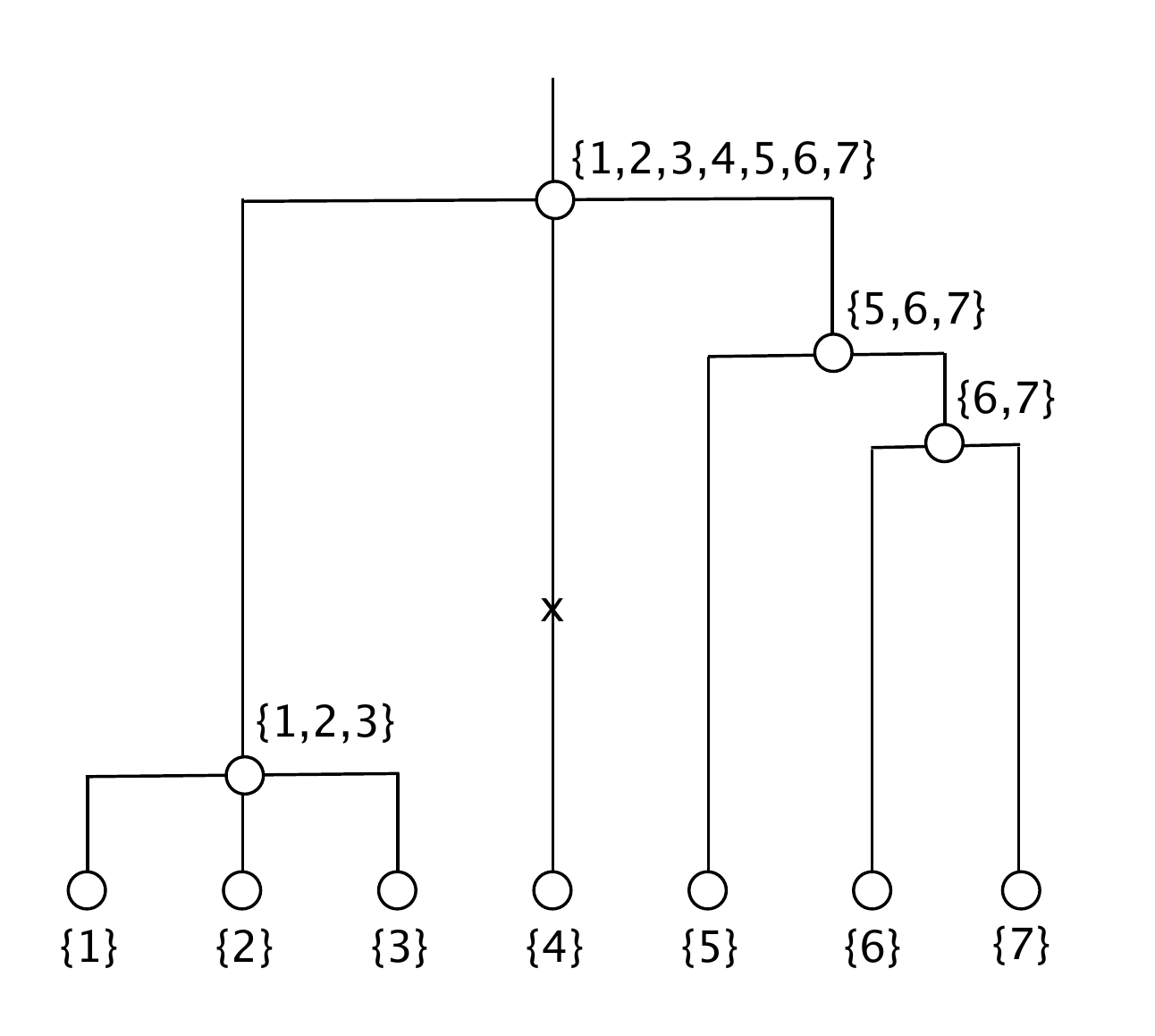}
\caption{A sample path of $\Pi_7$ with four collisions and one
transformation (denoted by ${\tt x}$) of the primary block $4$ into
the secondary block. The process $N_n$ has transitions $7\to 5\to 4\to 3\to 1$, thus $X_7=4$, $X_{7,2}=X_{7,3}=2$, the process $N^\ast_n$ has transitions $7\to 4\to 3\to 1\to 0$, thus $K_{7,1}=2,K_{7,2}=1,K_{7,3}=1$.}
\label{coal1}
\end{center}
\end{figure}

Similarly to
\cite{Gnedin+Iksanov+Marynych:2011} (formula (5.3)), the number of
collisions involving at most $\lfloor n^s\rfloor$ blocks can be decomposed as
\begin{equation}\label{eq:5_3_GIM}
X_n(s) \ = \ X_n^{\ast}(s) + D_{n}(s)\quad
\end{equation}
where, for $s\in[0,1]$, $X_n^{\ast}(s)$ is the number of collisions involving between 2 and $\lfloor n^s\rfloor$ blocks
such that at least two of them are primary, and $D_{n}(s)$ is the number of collisions involving at most $\lfloor n^s\rfloor$ blocks with at most one block being primary. Trivially, $D_{n}(s)\leq D_n$ where $D_n:=D_n(1)$. Furthermore,
$$
\sum_{k=2}^{\lfloor n^s\rfloor -Z_n}K_{n,k}\ \leq \ X_n^{\ast}(s)\ \leq \ \sum_{k=2}^{\lfloor n^s\rfloor }K_{n,k}\,,
$$
where the upper bound is obvious, and the lower bound follows since a decrement of $N^\ast_n$
of size at least two and at most $\lfloor n^s\rfloor -Z_n$ occurs by  a collision of size at most $\lfloor n^s\rfloor $ involving at least two primary particles.

Combining the aforementioned estimates and denoting $K_n(s):=\sum_{k=1}^{\lfloor n^s\rfloor }K_{n,k}$ for $s\in[0,1]$ we arrive at
\begin{equation}\label{eq:x_n_s_basic_estimate}
\sup_{s\in[0,1]}|X_n(s)-K_n(s)|\leq \sup_{s\in[0,1]}\left(\sum_{k=\lfloor n^s-Z_n\rfloor \vee 1}^{\lfloor n^s\rfloor }K_{n,k}\right)+D_n+K_{n,1}.
\end{equation}
The analogues of Theorems \ref{main1} and \ref{main2} have been proved for  $K_n(s)$, see  \cite{Alsmeyer+Iksanov+Marynych:2017} (Theorems 2.2 and 2.5). Hence, it is sufficient to show that for every fixed $\varepsilon>0$
\begin{equation}\label{eq:x_n_s_basic_estimate2}
\frac{\sup_{s\in[0,1]}\left(\sum_{k=\lfloor (n^s-Z_n)\vee 1\rfloor}^{\lfloor n^s \rfloor}K_{n,k}\right)+D_n+K_{n,1}}{(\log n)^{\varepsilon}}\overset{\mmp}{\to} 0,\quad n\to\infty.
\end{equation}
According to \cite{Gnedin+Iksanov+Marynych:2011}
the sequence $(D_n+K_{n,1})_{n\in\mn}$ is tight  (see
Lemma 5.1, Proposition 5.1 and also the proof of Theorem 5.1 therein). Appealing to
\eqref{eq:c_n_estimate}, we see that \eqref{eq:x_n_s_basic_estimate2} will follow from
\begin{equation}\label{eq:x_n_s_basic_estimate3}
\frac{\sup_{s\in[0,1]}\left(\sum_{k=\lfloor (n^s-c_0\log n)\vee 1\rfloor+1}^{\lfloor n^s\rfloor}K_{n,k}\right)}{(\log n)^{\varepsilon}}\overset{\mmp}{\to} 0,\quad n\to\infty.
\end{equation}

It remains to prove  \eqref{eq:x_n_s_basic_estimate3}. Let $(W_k)_{k\in\mn}$ be a sequence of independent copies of the random variable $W$ with distribution \eqref{eq:W_distribution},
$$
P_k \ := \ W_1\cdots W_{k-1}(1-W_k),\quad k\in\mn\quad\text{and}\quad \rho (x) \ := \ \#\{k\in\mn: P_k\geq 1/x\},\quad x\geq 0.
$$
We will use the following lemma borrowed from \cite{Alsmeyer+Iksanov+Marynych:2017} (formula (35) and Lemma 6.1).
\begin{lemma}\label{lemma:approx_gen_bound}
We have for every fixed $\varepsilon>0$
$$
\frac{ \sup_{s\in [0,1]}|K_{n}(s)-(\rho(n)-\rho(n^{(1-s)}))|}{(\log n)^{\varepsilon}} \overset{\mmp}{\to} 0,\quad n\to\infty.
$$
\end{lemma}
\vskip0.3cm
Set
$$
w_n(s)\ := \frac{\log^{+}(n^s-c_0\log n)}{\log n}\quad\text{and}\quad s_0(n):=\frac{\log (1+c_0\log n)}{\log n},\quad s\in [0,1]
$$
and note that
\begin{align}
\sup_{s\in[0,1]}\left|(1-s)\log n - (1-w_n(s))\log n\right|& = \sup_{s\in[0,1]}\left|s\log n-\log^{+}(n^s-c_0\log n)\right|\nonumber\\
&=\max\left\{s_0(n)\log n,\sup_{s\in[s_0(n),1]}\left|\log\left(1-\frac{c_0\log n}{n^s}\right)\right|\right\}\nonumber\\
&=\log(1+c_0\log n)\label{eq:loglog_estimate}.
\end{align}
Furthermore, \eqref{eq:x_n_s_basic_estimate3} is equivalent to
\begin{equation}\label{eq:x_n_s_basic_estimate4}
\frac{\sup_{s\in[0,1]}\left(K_n(s)-K_n(w_n(s))\right)}{(\log n)^{\varepsilon}}\overset{\mmp}{\to} 0,\quad n\to\infty,
\end{equation}
and applying Lemma \ref{lemma:approx_gen_bound} we see that \eqref{eq:x_n_s_basic_estimate4} follows if
one can show that
\begin{equation*}
\frac{\sup_{s\in[0,1]}\left|\rho(n^{1-w_n(s)})-\rho(n^{1-s})\right|}{(\log n)^{\varepsilon}}\overset{\mmp}{\to} 0,\quad n\to\infty.
\end{equation*}
By \eqref{eq:loglog_estimate} the left-hand side of the last formula satisfies
\begin{equation*}
\frac{\sup_{s\in[0,1]}\left|\rho(e^{(1-w_n(s))\log n})-\rho(e^{(1-s)\log n}))\right|}{(\log n)^{\varepsilon}}\leq \frac{\sup_{s\in[0,1]}(\rho(e^{s\log n +\log (1+c_0\log n)})-\rho(e^{s\log n}))}{(\log n)^{\varepsilon}}.
\end{equation*}
Letting $b(t):=\log(1+c_0t)$  the right-hand side of the latter inequality is bounded from above by
\begin{multline*}
\frac{\sum_{k=0}^{\lfloor b(\log n)\rfloor}\sup_{s\in[0,1]}(\rho(e^{s\log n +k+1})-\rho(e^{s\log n+k}))}{(\log n)^{\varepsilon}}\\
\leq ([b(\log n)]+1)\frac{\sup_{s\in[0,1]}(\rho(e^{(c+1)s\log n+1})-\rho(e^{(c+1)s\log n}))}{(\log n)^{\varepsilon}},
\end{multline*}
where the last estimate follows from  $b(t)\leq c_0t$ and
$$
\sup_{s\in[0,1]}(\rho(e^{s\log n +k+1})-\rho(e^{s\log n+k}))\leq \sup_{s\in[0,1]}(\rho(e^{(c_0+1)s\log n+1})-\rho(e^{(c_0+1)s\log n}))
$$
for $0\leq k\leq  b(\log n)$.
Finally, according to \cite{Alsmeyer+Iksanov+Marynych:2017} (Proposition 3.3)
$$
\frac{\sup_{s\in[0,1]}(\rho(e^{(c_0+1)s\log n+1})-\rho(e^{(c_0+1)s\log n}))}{((c\log n)^{\varepsilon/2}}\overset{\mmp}{\to} 0,\quad n\to\infty,
$$
whence \eqref{eq:x_n_s_basic_estimate3}.

\section{Other coalescents with dust}\label{Sect4}
We turn to the coalescents with ${\tt m}_{-2}=\infty$, yet ${\tt m}_{-1}<\infty$. In this case the collision times of $\Pi_\infty$  can be identified with the jump times of a subordinator $(S_t)_{t\geq 0}$  with the Laplace exponent
$$\Phi(z)= \log {\mathbb E}\, e^{-z S_1}=\int_{[0,\,1]}(1-(1-x)^z)x^{-2}\Lambda({\rm d}x).$$
The dust component has frequency $\exp(-S_t)$, that is for large $n$ the partition $\Pi_n(t)$ has about $n\,\exp(-S_t)$ primary singleton blocks.

We wish to approximate $X_{n,k}$ by $K_{n,k}$, the number of collisions which involve $k$ primary blocks (and possibly some secondary).
Let $D_{n,k}:= X_{n,k}-K_{n,k}$.  Our main tool is the following estimate.
\begin{lemma}\label{couplinglemma}
For $k=2,3,\ldots$
\begin{equation}\label{n20}
{\mathbb E}|D_{n,k}|\leq c_k\sum_{j=1}^n \left( \frac{\Phi(j)}{j}\right)^2,
\end{equation}
where $c_k$ is a positive constant.
\end{lemma}
\begin{proof}
We modify the argument for $X_n$ in
\cite{Gnedin+Iksanov+Marynych:2011} (Section 5.1). Let $X_{n,k+}$ be the number of collisions involving at least $k$ blocks, and let
$K_{n,k+}$ be the number of collisions involving at least $k$ primary blocks.
These variables are easier to compare, because $X_{n,k+}\geq K_{n,k+}$ and
$d_{n,k+}:={\mathbb E}(X_{n,k+}- K_{n,k+})\geq 0$.

Decomposing $X_{n,k+}$ at the first collision event of $\Pi_n$ we get
\begin{equation}\label{n17}
X_{1,k+}=0,~~X_{n,k+}\stackrel{{\rm d}}{=} \widehat{X}_{n-I_n,k+}+ \1_{\{I_n\geq k-1\}},~~~k\geq 2
\end{equation}
with the usual convention about the variables in the right-hand side, and $I_n$ having distribution \eqref{p-nk}. In the same way,
restricting the coalescent to the set of the primary blocks
\begin{equation}\label{n18}
K_{0,k+}=K_{1,k+}=0,~~~K_{n,k+}\stackrel{{\rm d}}{=} \widehat{K}_{n-I_n-1,k+}+ \1_{\{I_n\geq k-1\}},~~~k\geq 2.
\end{equation}
To make the right-hand sides  comparable, we need to adjust for $-1$ in recursion \eqref{n18}.
To that end, we focus on the singleton $\{n\}$ in the evolution of $\Pi_n$,
and identify $\Pi_{n-1}$ with the restriction of $\Pi_n$  to  $[n-1]$.
Thus $X_{n-1,k+}$ is realised as the count of  mergers of at least $k$ blocks for the restricted process.
The first collision involving $\{n\}$ occurs at some random time, say $\tau$, and for $\tau\geq t$ the  partitions  $\Pi_n(t)$ and $\Pi_{n-1}(t)$ have the same number of blocks, that is $N_n(t)=N_{n-1}(t)$ for $\tau\geq t$.
Now, if $\{n\}$ at time $\tau$ is merged together with $k-1$ other blocks
(a $k$-merger), then   $X_{n,k+}=X_{n-1,k+}+1$.
Otherwise $X_{n,k+}=X_{n-1,k+}$.
We see that $X_{n-1,k}$ and $X_{n,k}$  differ only in the case when the first collision taking $\{n\}$ is  a $k$-merger, hence involving at most
$k$ primary blocks. It follows that $X_{n,k+}\leq X_{n-1,k+}+ Y_{n,k}$, where $Y_{n,k}$ is the indicator of the event that the first collision with $\{n\}$ takes at most $k$ primary blocks.

From \eqref{n17} we now conclude that
\begin{equation}\label{n23}
X_{n,k+}\stackrel{{\rm d}}{\leq} X_{n-I_n-1,k+}+Y_{n-I_n,k}+ \1_{\{I_n\geq k-1\}},
\end{equation}
where the inequality is meant in the sense of stochastic order.
With  $y_{n,k}:={\mathbb E}Y_{n,k}={\mathbb P}\{Y_{n,k}=1\}$,
taking expectations in \eqref{n18} and \eqref{n23}, and subtracting the first relation from the second we obtain
$$d_{0,k+}=d_{1,k+}=0,~~d_{n,k+}\leq\sum_{j=1}^{n-1}p_{n,n-j}(d_{j-1,k+}+y_{j,k}),\quad n\geq2.$$ To evaluate $y_{n,k}$ we observe that this is the probability that in the partition of $[n]$ by the first event, the block containing element $n$ has size at most $k$. The expected total number of elements in such blocks is $\sum_{r=1}^{k} r\, {\mathbb E}K_{n,r}$, hence by exchangeability among $n$ primary blocks $$y_{n,k}=\frac{1}{n}\sum_{r=1}^{k} r\,{\mathbb E}K_{n,r}.$$
Following the same line  as in \cite{Gnedin+Iksanov+Marynych:2011} (Section 5.1)
we obtain for some positive constants $c_r, c', c_{k+}$ $${\mathbb E} K_{n,r}\leq c_r \Phi(n),$$
whence $y_{n,k}\leq c' \,\Phi(n)$ and therefore (see \cite{Gnedin+Iksanov+Marynych:2011})
$$d_{n,k+}\leq c_{k+}\sum_{j=1}^n  \left(\frac{\Phi(j)}{j}\right)^2.$$
Using $X_{n,k}=X_{n,k+}-X_{n,(k-1)+}$    and $K_{n,k}=K_{n,k+}-K_{n,(k-1)+}$ we
decompose the difference in question as $$D_{n,k}=(X_{n,k+}-K_{n,k+})- (X_{n,(k-1)+}-K_{n,(k-1)+}).$$
Since the differences in parantheses are both nonnegative, taking expectations and applying the triangle inequality
we obtain
$${\mathbb E}|D_{n,k}|\leq d_{n,k+}+d_{n,(k-1)+}.$$
The desired estimate \eqref{n20} easily follows with constant $c_k=c_{k+}+c_{(k-1)+}$.
\end{proof}

In \cite{Gnedin+Iksanov+Marynych:2011} we used a coupling with a subordinator to derive a limit law for $X_n$ under the condition of regular variation
\begin{equation}\label{eq:lambda_reg_var_dust}
\int_{[x,\,1]}y^{-2}\Lambda({\rm d}y)~ \sim~ x^{-\gamma}\ell(1/x),\quad x\to 0+
\end{equation}
for some $\gamma\in(0,1)$ and a function $\ell$ slowly varying at infinity. Specifically,
\begin{equation}\label{limXn}
\frac{X_n}{n^\gamma \ell(n)} \stackrel{{\rm d}}{\to} \Gamma(2-\gamma) {\cal E}_{\gamma},
\end{equation}
where the random variable
$${\cal E}_\gamma:=\int_0^\infty e^{-\gamma S_t}\,{\rm d}t$$
is known as an exponential functional of a subordinator. The following result is the extension for the collision spectrum.
\begin{thm}\label{main3}
If condition  {\rm ~(\ref{eq:lambda_reg_var_dust})} holds, then
as $n\to\infty$
\begin{equation}\label{eq:conv}
\frac{1}{n^{\gamma}\ell(n)}\left(X_{n,k}\right)_{k\ge 2}~\Longrightarrow~ \left(\frac{\gamma\Gamma(k-\gamma)}{k!}\,\,{\cal E}_\gamma \right)_{k\ge 2}
\end{equation}
weakly in the product space $\mr^\infty$.
Moreover, the convergence of joint moments holds: for all $m\ge 2$ and nonnegative integers  $q_2,\ldots,q_m$, as $n\to\infty$,
\begin{eqnarray}\label{eq:momconv}
\me\prod_{k=2}^m \left(\frac{X_{n,k}}{n^\gamma\ell(n)}\right)^{q_k}
\ \to\
\prod_{k=2}^m\left(\frac{\gamma\Gamma(k-\gamma)}{k!}\right)^{q_k}
\me\left({\cal E}_\gamma \right)^{q_2+\cdots+q_m}
=
\nonumber
\\
\prod_{k=2}^m\left(\frac{\gamma\Gamma(k-\gamma)}{k!}\right)^{q_k}~\prod_{j=1}^{q_2+\dots+q_m} \frac{j}{\Phi(\gamma j)}\,.
\end{eqnarray}
\end{thm}
\begin{proof}  Convergence \eqref{eq:conv} is concluded from the counterpart result for $(K_{n,k})_{k\geq 2}$
viewed in the context of $\Pi_\infty$. Indeed, by \cite{GPY1} (Theorem 4.1)
$$\frac{K_{n,k}}{n^\gamma\ell(n)}~\to~\frac{\gamma \Gamma(k-\gamma)}{k!} \,\,  {\cal E}_\gamma ~~~{\rm a.s.}$$
On the other hand, changing a variable and integrating by parts followed by an application of Karamata's Tauberian theorem \cite{BGT} (Theorem 1.7.1') give $\Phi(z)\sim \Gamma(1-\gamma) z^\gamma \ell(z)$ as $z\to\infty$.
Using this asymptotics along with \eqref{n20} readily yields ${\mathbb E}D_{n,k}=o(n^\gamma\ell(n))$, whence
$|K_{n,k}-X_{n,k}|/(n^\gamma \ell(n))\stackrel{\mathbb P}{\to}0$.

The assertion about the convergence of moments \eqref{eq:momconv} follows by dominated convergence from the analogous fact for $X_n$
proved in \cite{Haas+Miermont:2011}, and the familiar formula for moments of ${\cal E}_\gamma$ (see, e.g. \cite{GPY1} and references therein).
\end{proof}

Formal summation in \eqref{eq:conv} yields \eqref{limXn}. We see that with the same scaling the variables $X_n$, $X_{n,k}$'s all converge in distribution to multiples of  ${\cal E}_\gamma$.

Theorem \ref{main3} covers beta$(a,b)$-coalescents with $a\in (1,2)$. In this case \eqref{eq:lambda_reg_var_dust} holds
with $\gamma=2-a\in (0,1)$ and constant function $\ell(x)\equiv   1/  (\gamma{\rm B}(a,b))$. The related exponential functional is denoted by ${\cal E}_{2-a}$
(Table 1 in Section \ref{Sect1}). Extension to the case $\gamma=1$ is possible  for some slowly varying factors (for instance $\ell(z)=(\log z)^{-\theta}$ with $\theta>2$). A further extension concerns  the spectrum of  coalescent $(\Pi_n(t), t\in [0,T])$ with finite time horizon $T$; the assertions of Theorem \ref{main3} hold then with ${\cal E}_\gamma$ replaced by the incomplete exponential functional $\int_0^T \exp(-\gamma S_t){\rm d}t$. Another edge case is that of slow variation, where (\ref{eq:lambda_reg_var_dust}) holds with $\gamma=0$ and some unbounded function $\ell$.
In that case the series \eqref{n20} converges which implies by Lemma \ref{couplinglemma} that the sequence $(D_{n,k})_{ n\in{\mathbb N}}$ is tight for every $k$. Moreover, each ${\mathbb E}X_{n,k}$ is then of smaller order of growth than ${\mathbb E}X_{n}$. However, the limit laws for small block counts $K_{n,k}$ are only available for functions $\ell$ of logarithmic growth \cite{GPY2}, although there are plentiful results on $K_n$ \cite{GBslow, GIksslow} which have their counterparts for $X_n$ \cite{Gnedin+Iksanov+Marynych:2011} (Section 5.3). Thus we confine ourselves with the framework of \cite{GPY2}, assuming that
the characteristic measure satisfies
\begin{eqnarray}
\int_{[x,\,1]}y^{-2}\Lambda({\rm d}y)&=& |\log x| +c + O(x^{-\epsilon}),  ~~x\to 0+, \label{log1}\\
\int_{[x,\,1]}y^{-2}\Lambda({\rm d}y)&= & O((1-x)^{\epsilon}),~~x\to 1-
\label{log2}
\end{eqnarray}
for some constants $\epsilon>0$ and $c$.
Introduce the logarithmic moments
$$\nu_r=\int_{[0,\,1]}|\log (1-x)|x^{-2}\Lambda({\rm d}x), ~~r=1,2$$
(so $\nu_1={\mathbb E}S_1$ and $\nu_2={\rm Var}S_1$ for the corresponding subordinator).

Our next result follows from \cite{GPY2} (Theorem 15) and the discussion above.
\begin{thm}\label{gammacase} If  conditions {\rm(\ref{log1})} and  {\rm(\ref{log2})} hold, then
$$
  \left( \frac{X_{n,k}-(k\nu_1)^{-1}\log n}{\sqrt{\log n}}\right)_{k\geq 2}~\Longrightarrow~ \left(\mathcal{N}_k\right)_{k\geq 2}
$$
weakly in $\mr^{\infty}$, where the limit is a zero-mean Gaussian sequence with the covariance matrix
$$
\left(\frac{\nu_2}{\nu_1^3}\frac{1}{ij}+\delta_{i,j}\frac{1}{j\nu_1}\right)_{i,j\geq 2}.
$$
\end{thm}
Despite the seemingly limited scope, the theorem covers a number of important cases. For instance, $\Lambda({\rm d}x)= \frac{x^2(1-x)^{\theta-1}}{|\log(1-x)|}\1_{(0,1)}(x){\rm d}x$ is the case where $(S_t)_{t\geq 0}$ is the classic gamma-subordinator with parameter $\theta>0$ and the Laplace exponent $\Phi(z)=\log(1+z/\theta)$ (see \cite{Gnedin+Iksanov+Marynych:2011} for constants and the normal limit for $X_n$). Another example is the beta$(2,b)$-coalescent; in that case the logarithmic moments can be evaluated in terms of the Hurwitz zeta function $\zeta(z,b):=\sum_{i\geq 0}(i+b)^{-z}$
as $\nu_1=\zeta(2,b)$ and $\nu_2=2\zeta(3,b)$.

\section{Beta-coalescents without dust component}\label{Sect5}

The $\Lambda$-coalescents with  ${\tt m}_{-1}=\infty$ are very different from the coalescents with dust component, and, as such,
require other approaches. The most general available result on the total number of collisions $X_n$ states a stable limit distribution for the coalescents with characteristic measure satisfying $\Lambda([0,x])\sim c x^a$ for $x\to 0+$ (where $c>0,~a\in (0,1)$)
and a condition on the remainder of the expansion at $0$, see \cite{Gnedin+Yakubovich:2007} (Theorem 7).
This covers, in particular, all beta$(a,b)$-coalescents with $0<a<1$.
We also know that a stable limit for $X_n$ holds for beta$(1,b)$-coalescents \cite{Gnedin+Iksanov+Marynych+Moehle:2014}.

In what follows  we shall confine ourselves to the  family of  beta$(a,b)$-coalescents with $a\in (0,1]$.
The qualitative difference between the beta coalescents with $0<a<1$ and $a=1$ is that in the first case $\Pi_\infty(t)$ has finitely many blocks for all $t>0$ (the coalescent `comes down from infinity') and terminates in finite time, while in the second case the number of blocks always stays infinite.
\begin{thm}\label{main4}
Suppose $\Lambda$ is a {\rm beta}$(a,b)$-distribution with parameters  $a\in(0,1]$ and $b>0$.
Let
$$p_k^{(a)}:=\frac{(2-a)\Gamma(k+a-1)}{\Gamma(a)(k+1)!},~~ k\in\mn,$$
in particular, $p_k^{(1)}=(k^2+k)^{-1}$.
\begin{itemize}
\item[\rm (i)] If $0<a<1$, then as $n\to\infty$
$$
\left(\frac{X_{n,k}-p^{(a)}_{k-1} (1-a)n}{(1-a)n^{1/(2-a)}}\right)_{k\geq 2}~\Longrightarrow~ \left(p^{(a)}_{k-1}
\mathcal{S}_{2-a}\right)_{k\geq 2}
$$
weakly in $\mr^\infty$, where $\mathcal{S}_\alpha$,  for $\alpha\in (1,2)$, is an
$\alpha$-stable random variable with the characteristic function
$$u\mapsto\exp\{|u|^\alpha(\cos(\pi\alpha/2)+{\rm i}\sin(\pi\alpha/2){\rm sgn}(u))\}, ~u\in\mr.$$
\item[\rm(ii)] If $a=1$, then as $n\to\infty$
$$
\left(n^{-1}(\log n)^2X_{n,k}-p_{k-1}^{(1)}(\log n+\log\log n)\right)_{k\geq 2}~\Longrightarrow~ \left(p_{k-1}^{(1)}\mathcal{S}_1\right)_{k\geq 2}
$$
weakly in $\mr^\infty$, where $\mathcal{S}_1$ is a $1$-stable random variable with the characteristic function
$$u\mapsto\exp\Big({\rm i}u\log|u|-\frac{\pi}{2}|u|\Big)=({\rm i}u)^{{\rm i}u},~ u\in\mr.$$
\end{itemize}
\end{thm}
\begin{proof}
For $m\ge 2$ and $\beta_2,\ldots,\beta_m\in\mr$ we consider the linear combinations
$q^{(a)}:=\sum_{k=2}^m \beta_k p_{k-1}^{(a)}$ and $Z_n:=\sum_{k=2}^m \beta_k X_{n,k}$.
To apply the Cram\'er--Wold device we need to  prove that
\begin{equation}\label{eq:renewal_main_goal1}
\frac{Z_n-q^{(a)}(1-a)n}{(1-a)n^{1/(2-a)}}~\dod~ q^{(a)}
\mathcal{S}_{2-a},\quad n\to\infty
\end{equation}
in case (i), and that
\begin{equation}\label{eq:renewal_main_goal2}
   n^{-1}(\log n)^2Z_n - q^{(1)}(\log n+\log\log n)
~\dod~ q^{(1)}\mathcal{S}_{1}, \quad n\to\infty
\end{equation}
in case (ii). As before, we denote by $I_n$ the number of blocks involved in the first collision of $\Pi_n$ minus one. Decomposing at the first collision we obtain a stochastic recurrence
\begin{equation}\label{eq:z_recursion}
\quad Z_1=0,\quad Z_n\ \od\ \sum_{k=2}^m\beta_k\1_{\{I_n=k-1\}}+\widetilde{Z}_{n-I_n},\quad n\ge 2,
\end{equation}
where $\widetilde{Z}_k\od Z_k$ for every $k\in\mn$ and $(\widetilde{Z}_k)_{k\in\mn}$ is independent of $I_n$.

It is known (see \cite{Delmas} (Lemma 2.1) for $a\in (0,1)$ and
\cite{Drmota+Iksanov+Moehle+Roesler:2007} (p.~1409) for $a=1$) that
under the assumptions of Theorem \ref{main4} there exists a distributional limit
\begin{equation}\label{eq:ind_conv}
I_n \ \dod \ \xi,\quad n\to\infty,
\end{equation}
where $\xi$ is a random variable with distribution $\mmp\{\xi=k\}=p^{(a)}_k$, $k\in\mn$ and $\me \xi=(1-a)^{-1}$.

Consider an ordinary random walk $(S_j)_{j\in\mn_0}$ defined by $S_0:=0$ and $S_j:=\xi_1+\cdots+\xi_j$ for $j\in\mn$,
where $\xi_1,\xi_2,\ldots$ are independent copies of $\xi$,
and denote by $T_n:=\inf\{j\in\mn_0:S_j \ge n\}$, $n\in\mn_0$ the level $n$ first-passage time. For $n\in\mn_0$ and $k\in\mn$ let
$$J_{n,k} :=\sum_{j=1}^{T_n}\1_{\{\xi_j=k\}}
=\sum_{j\geq 1}\1_{\{\xi_j=k,S_{j-1}\leq n-1\}}$$ be the number of jumps of size $k$ before the random walk passes level $n$,
and set $Y_n:=\sum_{k=2}^{m}\beta_k J_{n, k-1}$. A standard conditioning argument yields the recursion
\begin{equation}\label{eq:y_recursion}
Y_0=0,\quad Y_n\ \od\ \sum_{k=2}^m\beta_k\1_{\{\xi=k-1\}}
+\widetilde{Y}_{n-\xi\wedge n},\quad n\in\mn,
\end{equation}
where $\widetilde{Y}_k\od Y_k$ for every $k\in\mn_0$ and
$(\widetilde{Y}_k)_{k\in\mn_0}$
is independent of $\xi$. Comparing \eqref{eq:z_recursion} and
\eqref{eq:y_recursion} and keeping in mind \eqref{eq:ind_conv} we may anticipate that if $Y_n$, properly centered and normalised,
converges in distribution, then the same holds for $Z_n$.
The subsequent proof of this intuition is split in two steps. The first step shows that relations \eqref{eq:renewal_main_goal1} and \eqref{eq:renewal_main_goal2} hold with $Y_n$ replacing $Z_n$. The second step derives from this the convergence for $Z_n$.

{\sc Step 1.} The representation
\begin{align*}
Y_{n+1}\ & = \sum_{k=2}^m\beta_k\sum_{j=1}^{T_{n+1}}\1_{\{\xi_j=k-1\}}
\ = \sum_{k=2}^m\beta_k\sum_{j\ge 1}\1_{\{\xi_j=k-1,\,S_{j-1}\leq n\}}\\
 & =  \sum_{j\ge 0}\1_{\{S_j\le n\}}\sum_{k=2}^m\beta_k\1_{\{\xi_{j+1}=k-1\}}
 \ =\ \sum_{j\ge 0}\eta_{j+1}\1_{\{S_j\le n\}},
\quad n\in\mn_0
\end{align*}
with $\eta_j:=\sum_{k=2}^m\beta_k\1_{\{\xi_j=k-1\}}$, $j\in\mn$, shows that $(Y_n)_{n\in\mn_0}$ has the same distribution as random process with
immigration in the sense of \cite{Iksanov+Marynych+Meiners:2017-1} (formula (1)) with random but constant response process
$\eta:=\sum_{k=2}^m\beta_k\1_{\{\xi=k-1\}}$. Thus,
\eqref{eq:renewal_main_goal1} with $Y_n$ replacing $Z_n$ follows from
 \cite{Iksanov+Marynych+Meiners:2017-1} (Theorem 2.4 applied with $u=1$,
$\alpha=2-a$, $h(t)=\me\eta=q_a$, $\rho=0$ and $p=1$).
In order to see that \eqref{eq:renewal_main_goal2} holds with $Y_n$ replacing by $Z_n$ decompose
$$
Y_{n+1}\ = \ \sum_{j\ge 0}\eta_{j+1}\1_{\{S_j\le n\}}
\ =\ (\me\eta)\sum_{j\ge 0}\1_{\{S_j\le n\}}
+ \sum_{j\ge 0} (\eta_{j+1}-\me\eta)\1_{\{S_j\le n\}}.
$$
By \cite{Iksanov+Moehle:2007} (Proposition 2), $n^{-1}(\log n)^2\sum_{j\ge 0}\1_{\{S_j\le n\}}-\log n-\log\log n\dod\mathcal{S}_1$
as $n\to\infty$. In view of Slutsky's lemma it would be enough to prove that
$n^{-1}(\log n)^2\sum_{j\ge 0}(\eta_{j+1}-\me\eta)\1_{\{S_j\le n\}}\overset{\mmp}{\to} 0$ as
$n\to\infty$. The latter convergence even holds in the mean-square sense, since
$$\me\Bigg(\sum_{j\ge 0}(\eta_{j+1}-\me\eta)\1_{\{S_j\le n\}}\Bigg)^2
=({\rm Var}\,\eta) \sum_{j\ge 0}\mmp\{S_j\le n\}=o(n)$$ as $n\to\infty$,
where the last estimate is a consequence of the elementary renewal theorem.

{\sc Step 2.} As in \cite{Gnedin+Iksanov+Marynych+Moehle:2014}
we use probability distances. For $q\in (0,1]$ and random variables $X$ and $Y$ having
finite $q$th moment the Wasserstein distance $d_q$ is defined via
$d_q(X,Y)=\inf\me|\widehat{X}-\widehat{Y}|^q$,
where the infimum is taken over all couplings $(\widehat{X},\widehat{Y})$
such that $\widehat{X}\od X$ and $\widehat{Y}\od Y$.
For properties of $d_q$ we refer the reader to \cite{Gnedin+Iksanov+Marynych+Moehle:2014} (Proposition 4.1).

To prove \eqref{eq:renewal_main_goal1} it suffices to show that
\begin{equation}\label{eq:renewal_main_goal11}
d_q(Y_n,Z_n) \ = \ o(n^{q/(2-a)}),\quad n\to\infty
\end{equation}
for some $q\in(0,1]$. Likewise, \eqref{eq:renewal_main_goal2} will follow from
\begin{equation}\label{eq:renewal_main_goal21}
d_q(Y_n,Z_n) \ = \ o(n^{q}(\log n)^{-2q}),\quad n\to\infty
\end{equation}
for some $q\in(0,1]$, see \cite{Gnedin+Iksanov+Marynych+Moehle:2014} (formulae (5.1) and (5.2)).

From $|J_{i,k} -J_{j,k}|\le |i-j|$ for $i,j\in\mn_0$ and $k\in\mn$ (the number of $k$-jumps while the random walks stays in $[i\wedge j-1, i\vee j-1)$ is dominated by the number of $1$-jumps which is $|i-j|$) it follows that
\begin{equation}\label{eq:y_sublinear}
|Y_i-Y_j|\ \le\ \|\beta\||i-j|,
\end{equation}
where $\|\beta\|:=\sum_{k=2}^m |\beta_k|$.

Let $(\hat{I}_n,\hat{\xi})$ be a coupling of $I_n$ and $\xi$ such that
$d_q(I_n,\xi\wedge n)=\me|\hat{I}_n-\hat{\xi}\wedge n|^q$. Further, let
$(\hat{Y}_j)_{j\in\mn_0}$ (respectively,
$(\hat{Z}_j)_{j\in\mn}$) be an arbitrary copy of
$(Y_j)_{j\in\mn_0}$ (respectively, $(Z_j)_{j\in\mn}$) independent of $(\hat{I}_n,\hat\xi)$.
Using recurrences \eqref{eq:z_recursion} and \eqref{eq:y_recursion} in combination with the inequality $|x+y|^q\leq |x|^q+|y|^q$ for $x,y\in \mr$ we infer
\begin{align*}
t_n & \ :=\ d_q(Y_n,Z_n)
\ \le\ \me\bigg|\sum_{k=2}^{m}\beta_k\1_{\{\hat{\xi}=k-1\}}+\hat{Y}_{n-\hat{\xi}\wedge
n}-\sum_{k=2}^m\beta_k\1_{\{\hat{I}_n=k-1\}}-\hat{Z}_{n-\hat{I}_n}\bigg|^q\\
& \le\ \me \bigg|\sum_{k=2}^m\beta_k(\1_{\{\hat{\xi}=k-1\}}-\1_{\{\hat{I}_n=k-1\}})\bigg|^q
+\me|\hat{Y}_{n-\hat{\xi}\wedge n}-\hat{Z}_{n-\hat{I}_n}|^q\\
& \le\ \me \bigg|\sum_{k=2}^{m}\beta_k(\1_{\{\hat{\xi}=k-1\}}-\1_{\{\hat{I}_n=k-1\}})\bigg|^q
+\me|\hat{Y}_{n-\hat{\xi}\wedge n}-\hat{Y}_{n-\hat{I}_n}|^q
+\me|\hat{Y}_{n-\hat{I}_n}-\hat{Z}_{n-\hat{I}_n}|^q.
\end{align*}
Passing to the infimum over all pairs
$((\hat{Y}_j,\hat{Z}_j))_{1\le j\le n-1}$ in the last summand leads to
$$
t_n\ \le\ c_n +\sum_{k=1}^{n-1}\mmp\{I_n=k\}t_{n-k},\quad n\ge 2
$$
with
$
c_n:=\me|\sum_{k=2}^{m}\beta_k(\1_{\{\hat{\xi}=k-1\}}-\1_{\{\hat{I}_n=k-1\}})|^q +\me|\hat{Y}_{n-\hat{\xi}\wedge
n}-\hat{Y}_{n-\hat{I}_n}|^q.
$
Applying \eqref{eq:y_sublinear} we obtain
\begin{align*}
c_n & \ \le\ \|\beta\|^q\big(\mmp\{\hat{I}_n\neq \hat{\xi}\wedge n\}+\me |\hat{I}_n-\hat{\xi}\wedge n|^q\big)\\
&  \ = \ \|\beta\|^q\big(\mmp\{|\hat{I}_n- \hat{\xi}\wedge n|^q \geq 1\}+\me |\hat{I}_n-\hat{\xi}\wedge n|^q\big)\\
&  \ \le\ 2\|\beta\|^q\me|\hat{I}_n-\hat{\xi}\wedge n|^q
\ = \ 2\|\beta\|^qd_q(I_n,\xi\wedge n).
\end{align*}
The proof of \eqref{eq:renewal_main_goal11} and
\eqref{eq:renewal_main_goal21} is completed along the lines of the argument
in \cite{Gnedin+Iksanov+Marynych+Moehle:2014} (the bottom of p.~504).
\end{proof}

The instance $b=1$ in part (ii) of the theorem is the asymptotics of the
collision spectrum for the Bolthausen--Sznitman coalescent. Note that
neither the limit law nor the scaling/centering constants depend on $b$. This suggests that the stable limit for the
collision spectrum in part (i) also holds in the more general setting of \cite{Gnedin+Yakubovich:2007}.

\noindent \textbf{Remark}. We conjecture that in the situation of Theorem \ref{main4} the moments of $X_{n,k}$
differ only little from those of $p_{k-1}^{(a)}X_n$ if $n$ is sufficiently
large. While this is open for $a\in (0,1)$, for $a=1$ an adaptation of the method of sequential
approximation, see \cite{Gnedin+Iksanov+Marynych+Moehle:2014} (Section 5.3),
shows that, for every $k\ge 2$ and $j\in\mn$,
\begin{equation} \label{eq:expansion}
\me X_{n,k}^j\ =\ \left(p_{k-1}^{(1)}\frac{n}{\log n}\right)^j
\left(1+\frac{m_j}{\log n}+O\left(\frac{1}{\log^2n}\right)\right),
\qquad n\to\infty,
\end{equation}
where the sequence $(m_j)_{j\in\mn_0}$ is recursively defined via
$m_0:=0$ and $m_j:=m_{j-1}+\kappa_j/j$ for $j\in\mn$, with
$\kappa_j:=(j+b-1)\Psi(j+b)+j-(b-1)\Psi(b)$, $j\in\mn$. Here, $\Psi$
denotes the logarithmic derivative of the gamma function. By
\cite{Gnedin+Iksanov+Marynych+Moehle:2014} (Theorem 3.2), \eqref{eq:expansion}
coincides with the second order expansion of the $j$th moment of
$p_{k-1}^{(1)}X_n$.

\vspace{5mm}

\noindent {\bf Acknowledgement}. A. Marynych was supported by the return fellowship of the Alexander von Humboldt Foundation. The publication contains the results of studies conducted by President's of Ukraine grant for competitive projects $\Phi$70 of the State Fund for Fundamental Research.

\end{document}